\documentclass[a4paper]{article}
\usepackage{amssymb}
\usepackage{amsmath}
\usepackage{amsthm}
\usepackage{hyperref}
\usepackage{enumerate}
\usepackage{mathtools}
\usepackage{url}
\usepackage[T1]{fontenc}
\usepackage{tikz}

\usetikzlibrary{patterns}

\newtheorem{Theorem} {Theorem} [section]
\newtheorem{Proposition} [Theorem] {Proposition}
\newtheorem{Assumption} [Theorem] {Assumption}

\newtheorem{Lemma} [Theorem] {Lemma}
\newtheorem{Corollary} [Theorem] {Corollary}

\newtheorem{Notation} [Theorem] {Notation}

\newcommand{\Ff}{{\mathbb F}}

\newcommand{\Rf}{{\mathbb R}}

\newcommand{\cL}{{\mathcal L}}

\newcommand{\cP}{{\mathcal P}}

\newcommand{\PG}{\mathrm{PG}}

\newcommand{\wt}{\mathrm{wt}}

\newcommand{\rank}{\mathrm{rank}}

\newcommand{\<}{\langle}
\renewcommand{\>}{\rangle} 
\renewcommand{\phi}{\varphi} 
\newcommand{\gauss}[2]{\genfrac{[}{]}{0pt}{}{#1}{#2}}

\title{The classification of Boolean degree $1$ functions in high-dimensional finite vector spaces}
\author{
 Ferdinand Ihringer
}
\date{26 May 2024}

\begin{document}
\maketitle

\begin{abstract}
  We classify the {\it Boolean degree $1$ functions}
  of $k$-spaces in a vector space of dimension $n$
  (also known as {\it Cameron-Liebler classes})
  over the field with $q$ elements for $n \geq n_0(k, q)$.
  This also implies that two-intersecting sets with
  respect to $k$-spaces do not exist for $n \geq n_0(k, q)$.
  Our main ingredient is the Ramsey theory for geometric lattices.
\end{abstract}


\section{Introduction}

A (real) Boolean degree $d$ function on the $n$-dimensional
hypercube $\{0,1\}^n$ is a real multivariate degree $d$ polynomial
$f: \{ 0,1 \}^n \rightarrow \{ 0, 1 \}$.
The classification of Boolean degree $1$ functions on the hypercube
is easy. A short calculation shows that $f$ depends on at most one
variable, that is either $f(x) = 0$, $f(x) = 1$, $f(x) = x_i$,
or $f(x) = 1 - x_i$ for some $i \in \{ 1, \ldots, n \}$.
For a set $A$, let $\binom{A}{b}$ denote the family of
$b$-subsets of $A$.
The {\it Johnson scheme} (or {\it slice of the hypercube})
$J(n, k)$ consists of $\binom{\{1, \ldots, n \}}{k}$, cf.~\cite{FI2019}.
Boolean degree $d$ functions on $J(n, k)$ are defined
the same way as on the hypercube $\{ 0, 1 \}^n$,
namely as a degree $d$ polynomial
$f: \binom{\{1, \ldots, n \}}{k} \rightarrow \{ 0, 1 \}$.
The existence of a Boolean degree $d$ function $f$
is equivalent to the existence of a weight function $\wt: \binom{\{ 1, \ldots, n \}}{d}
\rightarrow \Rf$ such that $\wt(S) := \sum_{T \in \binom{S}{d}} \wt(T)
\in \{ 0, 1 \}$ for all $S \in \binom{\{ 1, \ldots, n \}}{k}$.
Boolean degree $1$ functions on the Johnson scheme are
essentially the same as on the hypercube and have been classified
multiple times, each of \cite{Filmus2016,FI2019,Meyerowitz1992}
contains a (reasonably) short proof. Boolean degree $1$ functions have also been
classified for various other posets,
most prominently for the symmetric group \cite{EFP2011}.
See \cite{DFLLV2020} for a unified presentation of some techniques.

Our aim here is the classification of Boolean degree $1$
functions in a natural $q$-analog of the
Johnson scheme $J(n, k)$, the {\it Grassmann scheme} $J_q(n, k)$.
For a vector space $A$, let $\gauss{A}{b}$ denote the family
of $b$-subspaces of $A$. The Grassmann scheme $J_q(n, k)$
consists of $\gauss{V}{k}$, where $V$ is an $n$-dimensional
vector space over the field with $q$ elements.
A Boolean degree $d$ function on $J_q(n, k)$ corresponds
to a weighting $\wt: \gauss{V}{d} \rightarrow \Rf$
such that $\wt(S) := \sum_{T \in \gauss{S}{d}} \wt(T) \in \{ 0, 1 \}$
for all $S \in \gauss{V}{k}$.
We will use the following convention throughout the text:
\begin{Notation}\label{not:boolean_to_set}
  Identify a Boolean function $f: \gauss{V}{k} \rightarrow \{ 0, 1 \}$
  with the family $Y$ of $k$-spaces $S$ of $V$ which satisfy $f(S) = 1$.
\end{Notation}
Boolean degree $1$ functions on $J_q(n, k)$ are usually
called {\it Cameron-Liebler class-es} (or {\it Cameron-Liebler line classes} if $k=2$).
In the corresponding literature, Cameron-Liebler classes are commonly seen as
sets.
The following lists the {\it trivial examples} for a Boolean degree $1$ function
on $J_q(n, k)$ for $n \geq 2k$:
\begin{enumerate}[(I)]
 \item The empty set.
 \item All $k$-spaces through a fixed $1$-space $P$, a {\it point-pencil}.
 \item All $k$-spaces in a fixed hyperplane $H$ of $V$.
 \item The union of the previous two examples when $P \nsubseteq H$.
 \item The complement of any of the examples (I) to (IV).
\end{enumerate}
Let $B(n, k, q)$ denote the statement that
all examples in $V$ (for given $n, k, q$) are trivial examples.
In 1982, Cameron and Liebler conjectured that $B(4, 2, q)$ holds,
see \cite{CL1982}.
Bruen and Drudge disproved $B(4, 2, q)$ for all odd $q$
in \cite{BD1999,Drudge1998} by providing another
example. Nowadays examples disproving
$B(4, 2, q)$ for $q=4$, $q \equiv 1 \pmod{2}$,
and $q \equiv 2 \pmod{3}$ with $q>2$ are known,
see \cite{DeBeule2016,FMRXZ2021,FMX2015}.
In contrast to these results, $B(4, 2, 2)$ and
$B(n, k, q)$ hold for all $n \geq 2k$ when $q \in \{ 2, 3, 4, 5 \}$,
cf.~\cite{Drudge1998,FI2019,GM2014,GM2018,GP2005,Matkin2018}.
In \cite{FI2019}, the authors conjectured that there exists
an $n_0(k, q)$ such that $B(n, k, q)$ is true for all $n \geq n_0(k, q)$.
This is our main result.

\begin{Theorem}\label{thm:main}
  Let $q$ be a prime power and let $k$ be some integer, where $k \geq 2$.
  Then there exists a number $c_0(k, q)$ such that
  for all $\max(k, n-k) \geq c_0(k, q)$ all Boolean degree $1$
  functions on $k$-spaces in an $n$-dimensional vector space
  over a finite field with $q$ elements are trivial.
\end{Theorem}

This answers the question on the existence of non-trivial Boolean degree $1$
functions in the setting that is interesting from the point of
view of complexity measures in theoretical computer science,
cf.~\cite{DFLLV2020}.

As a corollary of our main result, we show that
so called \textit{two-intersecting sets} with respect to $k$-spaces do not
exist when $n \geq n_0(k, q)$. Following Notation \ref{not:boolean_to_set},
the \textit{size of a Boolean function} $Y$ refers to $|Y|$.
For the regime that $n$ is small, we also improve a gap result for
the possible sizes of non-trivial Boolean degree $1$ functions.

Our main ingredients are earlier work by Drudge \cite{Drudge1998}
and Metsch \cite{Metsch2010} on Cameron-Liebler sets
as well as the celebrated general Ramsey theorems for projective and
affine spaces due to Graham, Leeb, and
Rothschild \cite{GLR1972}. We refer to
\cite{FY2023} for recent explicit bounds on these Ramsey numbers.

A {\it $\delta$-coloring} of a set $A$ is a map $\Delta: A \rightarrow B$
with $|B| = \delta$.
Two elements $a,b \in A$ have the same color
if $\Delta(a) = \Delta(b)$.
Corollary 2 in \cite{GLR1972} implies the following.

\begin{Theorem}[Graham, Leeb, Rothschild (1972)]\label{thm:GLR_proj}
  Let $q$ be a prime power.
  For every pair of positive integers $\delta$ and $k$,
  there exists a positive integer $\tilde{n}_0(k, \delta, q)$ such that for
  every vector space $V$ of dimension $n \geq \tilde{n}_0(k, \delta, q)$
  over a finite field with $q$ elements
  and any $\delta$-coloring of $\gauss{V}{1}$,
  there exists a $K \in \gauss{V}{k}$ such that all elements of $\gauss{K}{1}$
  have the same color.
\end{Theorem}

As we can identify the $(k-1)$-spaces of an affine space
of dimension $n-1$ over some field $\Ff$
with the $k$-spaces of a vector space of dimension $n$ over $\Ff$
which are not contained in some fixed hyperplane $H$,
Corollary 3 in \cite{GLR1972} implies the following.

\begin{Theorem}[Graham, Leeb, Rothschild (1972)]\label{thm:GLR_aff}
  Let $q$ be a prime power.
  For every pair of positive integers $\delta$ and $k$,
  there exists a positive integer $\bar{n}_0(k, \delta, q)$ such that for
  every vector space $V$ of dimension $n \geq \bar{n}_0(k, \delta, q)$
  over a finite field with $q$ elements,
  every hyperplane $H$ of $V$,
  and any $\delta$-coloring of $\gauss{V}{1} \setminus \gauss{H}{1}$,
  there exists a $K \in \gauss{V}{k} \setminus \gauss{H}{k}$
  such that all elements of $\gauss{K}{1} \setminus \gauss{H}{1}$
  have the same color.
\end{Theorem}

Note that $\bar{n}_0(k+1, \delta, q) > \bar{n}_0(k, \delta, q)$.\footnote{%
Let a $\delta$-coloring $\Delta$ be a witness for $\bar{n}_0(k, \delta, q)$,
that is, the coloring $\Delta: \gauss{V}{1} \setminus \gauss{H}{1} \rightarrow \{ 1, \ldots, \delta \}$, where $\dim(V) = \bar{n}_0(k, \delta, q)-1$, does not contain
a monochromatic $k$-space. Embed $V$ as a hyperplane
in some vector space $V'$ with $\dim(V') = \bar{n}_0(k, \delta, q)$.
Let $P$ be a point of $V'$ not in $V$.
Define a new coloring $\Delta'$ by
\[
  \Delta'(Q) = \begin{cases}
                \Delta(\< P, Q \> \cap V) & \text{ if } P \neq Q,\\
                1 & \text{ if } P = Q.
              \end{cases}
\]
As $\Delta$ does not contain monochromatic $k$-spaces,
$\Delta'$ does not contain a monochromatic $(k+1)$-space.}

The connection between these Ramsey-type results
for vector spaces and affine spaces derives from the fact that
the weights of a Boolean degree $1$ function on the vector space $V$
give us a coloring of the points of $V$.

Let us also state Voigt's
more general version \cite[Satz 3.3]{Voigt1978}, cf.~\cite[Theorem 1.9]{PV1981},
which includes the classical Ramsey theorem
as well as Ramsey theorems for projective
and affine spaces.
For any prime power $q$, let $R\cL(q)$ denote the class
of (ranked) lattices generated by a subset of $1$-spaces
for some finite vector space $V$ over the field with $q$
elements. For a coloring of all elements of the (sub)lattice $L \in R\cL(q)$,
we call $S$ in $L$ {\it monochromatic}
if all pairs of subspaces of $T_1, T_2 \subseteq S$
in $L$ with $\dim(T_1) = \dim(T_2)$ have the same color.

\begin{Theorem}[Voigt (1978)]\label{thm:voigt_ramsey}
  Let $q$ be a prime power.
  For every pair of positive integers
  $\delta, k$ there exists a positive integer $n_0(k, \delta, q)$
  such that for every $L \in R\cL(q)$ with $\rank(L) \geq n_0(k, \delta, q)$
  and every $\delta$-coloring of $L$ there exists
  a monochromatic element $K \in L$ with $\dim(K) = k$.
\end{Theorem}

This result seems to have been overlooked.
For instance, Mazzocca and Tallini
seem to have been unaware of Voigt's result while
writing the final remark in \cite{MT1985}.
We will discuss this in more detail in Section \ref{sec:mt}.

\section{Preliminaries}

We use projective notation and algebraic dimensions,
so we call $1$-spaces {\it points} and $2$-spaces {\it lines}.
Here $q$ is always a prime power.
Put $[a] := \frac{q^a-1}{q-1}$.
For $a \geq 1$, $[a]$ denotes the number of points in an $a$-space.
We will use repeatedly that a line has $q+1$ points.
The number of $b$-spaces in an $a$-space is given by the
$q$-binomial coefficient
\[
  \gauss{a}{b} := \prod_{i=1}^b \frac{[a-i+1]}{[i]}.
\]
Note that we always fix $q$, so there is no ambiguity.
Let $x_P^+$ be the indicator function of $P$.
If we write $f(x) = \sum_{P \in \gauss{V}{1}} c_P x_P^+$
as a homgeneous degree $1$ polynomial
as in \cite{FI2019}, then $\wt(P)$ is the coefficient $c_P$ of $x_P^+$.
Note that $f$ determines $\wt$ and vice-versa.

\begin{Proposition}\label{prop:wt_bnds_gen}
  Let $n \geq 2k \geq 4$.
  Let $Y$ be a Boolean degree $1$ function of $J_q(n, k)$ of size $x\gauss{n-1}{k-1}$.
  Then, for all $P \in \gauss{V}{1}$,
 \begin{align}
    0 \leq \wt(P) + \frac{[k-1]}{[n-1]} \left( x - \wt(P) \right) \leq 1.
    \label{eq:wt_bnd}
 \end{align}
\end{Proposition}
\begin{proof}
 Write $w = \wt(P)$.
 As one point lies on $\gauss{n-1}{k-1}$ $k$-spaces, we find that
 \[
    \sum_{Q \in \gauss{V}{1}} \wt(Q) \gauss{n-1}{k-1} = \sum_{S \in \gauss{V}{S}} \wt(S) = |Y| = x \gauss{n-1}{k-1}.
 \]
 Hence, the average weight of a point distinct from $P$
 is $\frac{x-w}{[n]-1}$.
 Hence, the average weight of a $k$-space through $P$ is
 \[
  w + \frac{[k]-1}{[n]-1} (x-w) = w + \frac{[k-1]}{[n-1]} (x-w).
 \]
 This number is at least $0$ and at most $1$.
\end{proof}

Drudge observed in Lemma 6.1 and Lemma 6.3 of his PhD thesis
that if a trivial Boolean degree $1$ function
contains a point-pencil locally, then it is trivial
and contains a point-pencil globally, see \cite{Drudge1998}.

\begin{Proposition}[Drudge (1998)]\label{prop:drudge_easy_ind}
  Let $Y$ be a Boolean degree $1$ function on $J_q(n, 2)$.
  If the restriction of $Y$ to some $4$-space corresponds
  to Example (II), respectively, Example (IV), then $Y$
  corresponds to Example (II), respectively, Example (IV).
\end{Proposition}

We need the following gap result due to Metsch \cite{Metsch2010}.

\begin{Proposition}[Metsch (2010)]\label{prop:metsch_bnds}
  Let $Y$ be a non-trivial Boolean degree $1$ function
  of size $x (q^2+q+1)$ on $J_q(4, 2)$.
  Then $x \geq q+1$.
\end{Proposition}

\begin{Corollary}\label{cor:bnds_wt_n_eq_4}
  Let $Y$ be a Boolean degree $1$ function on $J_q(4, 2)$.
  \begin{enumerate}[(i)]
   \item If a point has weight larger than $\frac{q}{q+1}$,
  then $Y$ is trivial and contains a point-pencil.
  \item If a point has weight less than $-\frac{q-1}{q+1}$,
  then $Y$ is trivial and its complement contains
  a point-pencil.
  \end{enumerate}
\end{Corollary}
\begin{proof}
  By Proposition \ref{prop:metsch_bnds}, if the assertion in (i)
  is false, then $x \geq q+1$.
  Rearranging \eqref{eq:wt_bnd} for $\wt(P)$ shows
  that $\wt(P) \leq \frac{q}{q+1}$ if $Y$ is non-trivial.
  Similarly, by Proposition \ref{prop:metsch_bnds}, if the assertion
  in (ii) is false, then $x \leq q^2-q$.
  Rearranging \eqref{eq:wt_bnd} for $\wt(P)$ shows
  that $\wt(P) \geq -\frac{q-1}{q+1}$ if $Y$ is non-trivial.
\end{proof}

Lastly, we will use the following well-known facts.

\begin{Lemma}\label{lem:full_rk}
  The point-line incidence matrix of a $3$-dimensional vector space
  over a finite field has full rank.
\end{Lemma}
\begin{proof}
  Let $C$ denote the $(q^2+q+1)\times(q^2+q+1)$ point-line incidence matrix
  and let $q$ denote the order of the finite field.
  Then $C^T C = q I + J$, where $I$ denote the identity matrix
  and $J$ denotes the all-ones matrix. As $q \geq 2$, $C^T C$
  has full rank. Hence, $C$ has full rank.
\end{proof}

\begin{Lemma}\label{lem:pt_and_coline}
  For $n \geq 3$, let $V$ be an $n$-dimensional vector space.
  Let $P$ be a point and $C$ be an $(n-2)$-space
  with $P \nsubseteq C$. Then there exists a
  line through $P$ which has no point in $C$.
\end{Lemma}
\begin{proof}
  The quotient space $\< P, C \>/P$ is an $(n-2)$-dimensional vector
  space, and $V/P$ is an $(n-1)$-dimensional vector space.
  Hence, there exists a point in $V/P$ which is not in $\< P, C \>/P$.
  Hence, there exists a line in $V$ through $P$ which
  does not contain a point of $C$.
\end{proof}

\section{Proof of Theorem \ref{thm:main}}

Theorem 4.15 of \cite{FI2019} shows that
$B(n, k, q)$ implies $B(n+1, k+1, q)$ and
$B(n+1, k, q)$. Hence, to prove Theorem \ref{thm:main},
it is enough to show that $B(n, 2, q)$ holds for
all $n \geq n_0(q)$ for some constant $n_0(q)$. Then the assertion follows
for $\min(k, n-k) \geq 2$ and $\max(k, n-k) \geq n_0(q)+k$.
Hence, throughout this section, we will make the following assumption.
\begin{Assumption}
  Assume that $Y$ is a Boolean degree $1$ function on $J_q(n, 2)$.
\end{Assumption}

\noindent
It is also convenient to group points by their weight.
\begin{Notation}
 Denote the set of points with weight $w$ by $\cP(w)$.
\end{Notation}

\noindent
Now we are ready to prove Theorem \ref{thm:main}: First we will show in Lemma \ref{lem:classical_weights_are_good} that
if we restrict $w$ to certain weights, then Theorem \ref{thm:main} holds.
Then we will show in Proposition \ref{prop:rec_ram} that for $n$ sufficiently large, Lemma \ref{lem:classical_weights_are_good} applies.

\begin{Lemma}\label{lem:classical_weights_are_good}
  Let $n \geq 4$.
  If $Y$ has point weights only in $[-1, -\frac{q-1}{q+1}) \allowbreak
  \cup \{-\frac{1}{q^2+q}, 0, \allowbreak \frac{1}{q+1}, \frac{1}{q}\} \allowbreak \cup (\frac{q}{q+1}, 1]$,
  then $Y$ is trivial.
\end{Lemma}
\begin{proof}
  If a point has a weight in $[-1, -\frac{q-1}{q+1})$
  or $( \frac{q}{q+1}, 1 ]$, then the assertion follows
  from Proposition \ref{prop:drudge_easy_ind} and Corollary \ref{cor:bnds_wt_n_eq_4}.
  Hence, we can assume that all weights are in $\{ -\frac{1}{q^2+q}, 0, \frac{1}{q+1}, \frac{1}{q}\}$.

  A line $L$ through a point with weight $0$ necessarily has all
  its points in $\cP(0) \cup \cP(\frac{1}{q})$ and $L$ intersects
  $\cP(0)$ in either $1$ or $q+1$ points:
  Say that $L$ has $\alpha$ points in $\cP(-\frac{1}{q^2+q})$.
  Then $\wt(L) \leq \alpha \cdot \left(-\frac{1}{q^2+q}\right) + (q-\alpha) \cdot \frac{1}{q}$.
  If $\wt(L) = 1$, then we obtain $\alpha \leq 0$.
  Hence, $\wt(L) = 1$ if and only if $q$ points of $L$
  are in $\cP(\frac{1}{q})$. Hence, $\wt(L) = 0$ if $\alpha > 0$.
  At least one point on $L$ must have positive weight. Hence, $\alpha \leq q-1$.
  Thus, $0 = \wt(L) \geq \alpha \cdot \left(-\frac{1}{q^2+q}\right) + \frac{1}{q+1} \geq \frac{1}{q^2+q} > 0$,
  also impossible.
  Hence, if $\wt(L) = 0$, then $L$ is contained in $\cP(0)$.

  The previous paragraph shows that
  the set $\cP(0)$ forms a subspace $S$.
  Also note that $\cP(\frac{1}{q})$ cannot contain a line
  as $(q+1) \cdot \frac{1}{q} \not\in \{ 0, 1 \}$.
  If $\dim(S) = n$, then we have the complement of Example (I).
  If $1 \leq \dim(S) \leq n-1$, then pick
  any $(\dim(S)+1)$-space $R$ through $S$.
  By the argument above, $R \setminus S \subseteq \cP(\frac{1}{q})$.
  As this holds for any such $R$, all points not in $S$
  have weight $\frac{1}{q}$. But $\cP(\frac{1}{q})$ cannot
  contain a line, so $S$ must be a hyperplane.
  Hence, if $\cP(0)$ is not empty or the whole space,
  then $\cP(0)$ is a hyperplane and $\cP(\frac{1}{q})$ is its complement.
  Hence, $Y$ is the complement of Example (III).

  Analogously, if $\cP(\frac{1}{q+1})$ is not empty or the whole space, then
  $\cP(\frac{1}{q+1})$ is a hyperplane and $\cP(\frac{-1}{q^2+q})$
  is its complement.
  To see this, consider the complement $\overline{Y}$ of $Y$,
  that is, $\overline{Y} = 1 - Y$:
  Let $\overline{\cP}(w)$ denote the points with weight $w$ with respect to $\overline{Y}$.
  The constant function $1$ has weight $\frac{1}{q+1}$ for all points.
  Hence, $P \in \cP(w)$ if and only if $P \in \overline{\cP}(\frac{1}{q+1} - w)$.
  In particular, $\cP(\frac{1}{q+1}) = \overline{\cP}(0)$
  and $\cP(\frac{1}{q}) = \overline{\cP}(-\frac{1}{q^2+q})$.
  Thus, we are in the previous situation with $\overline{Y}$ in the role of $Y$.
  In this case, $Y$ is either Example (I) or Example (III).

  It remains to exclude the case that all weights are in
  $\cP(\frac{-1}{q^2+q}) \cup \cP(\frac{1}{q})$. But then the weight of
  a line is $\frac{-\alpha}{q^2+q} + \frac{q+1-\alpha}{q} = 1+\frac{q+1-(q+2)\alpha}{q^2+q} \in \{ 0, 1\}$
  for some $\alpha \in \{ 0, \ldots, q+1 \}$ which is impossible.
\end{proof}

Next we show that for $n$ large enough, Lemma \ref{lem:classical_weights_are_good}
is always applicable.

\begin{Lemma}\label{lem:aff_weights}
  Let $Y$ be a Boolean degree $1$ function on $J_q(n, 2)$, where $n \geq 2$,
  such that there exists a hyperplane $H$ of the $n$-space $V$ such that all
  points in $H$ have weight $0$ and all points not in $H$
  have constant weight $w$. Then $w \in \{ 0, \frac1q\}$.
\end{Lemma}
\begin{proof}
  As $n \geq 2$, there exists
  a point $P$ in $H$
  and a point $Q$ not in $H$.
  Precisely one point of $\<P, Q\>$ is in $H$, namely $P$,
  and has weight $0$. The $q$ remaining points have weight $w$,
  so $qw$ is either $0$ or $1$.
\end{proof}

\begin{figure}
\centering
    \begin{tikzpicture}[
        scale=0.5,
        line width=1pt,
        fill=black,
        every node/.style={inner sep=0pt, outer sep=0pt},
        every circle/.style={radius=2.5pt}
    ]

    \node[name=Tip1,label=200:{ $T_{3}$}] at (1,0) {};
    \node[name=Ti] at (0,6) {};
    \node[name=P,outer sep=1pt,label=40:{ $P$}] at (12,3) {};
    \node[name=S1,label=40:{ $S_{2}^1$}] at (8,10) {};
    \node[name=S2,label=40:{ $S_{2}^2$}] at (9.2,8) {};
    \node[name=H,label=40:{ $H_2$}] at (10.6,5.5) {};

    \path (Ti) edge (Tip1);
    \path (Ti) -- (Tip1) node [midway,label=200:{ $T_{2}$}] {};
    \draw (Ti) -- (P) -- (Tip1) -- (Ti);
    \path (Tip1) edge (P);
    \path (Ti) edge (S1);
    \path (Tip1) edge (S1);
    \path (Ti) edge (S2);
    \path (Tip1) edge (S2);
    \path (Ti) edge (H);
    \path (Tip1) edge[line width=2.5pt] (H);

    \fill[shift=(P)] circle[radius=3.2pt];
    \fill[shift=(Ti)] circle[radius=1pt];
    \fill[shift=(Tip1)] circle[radius=3.2pt];

    \fill[line width=2.5pt, pattern color=black!50, dotted, pattern=dots] (8,10) -- (1,0) -- (12,3);
    \draw[line width=2.5pt, dotted] (Tip1) -- (S1);
    \draw[line width=2.5pt, dotted] (Tip1) -- (S2);
    \end{tikzpicture}

    \caption{An illustration of the induction in the proof of Proposition \ref{prop:rec_ram}
    for $i=2$. The dotted lines correspond to $S_3^1$ and $S_3^2$, the thick
    line corresponds to $S_3^3$, and the dotted area corresponds to
    $V_3 = \< P, S_3^3\>$.
    }
    \label{fig:rec}
\end{figure}
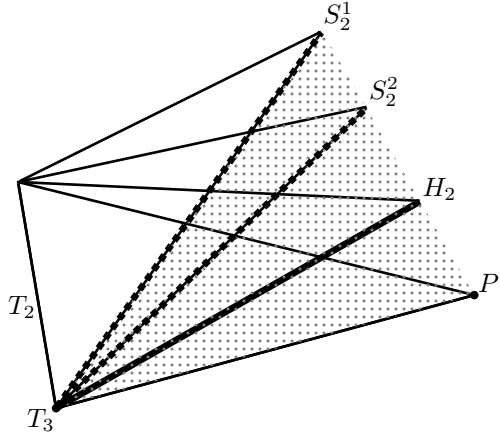

\begin{Proposition}\label{prop:rec_ram}
  Let $q$ be a prime power. Then there exists a constant $n_0(q)$
  such that for all $n \geq n_0(q)$ all points
  of a Boolean degree $1$ function on $J_q(n, 2)$ have weights
  in $\{ -\frac{q}{q+1}, -\frac{q-1}{q}, -\frac{1}{q^2+q}, 0, \frac{1}{q}, \frac{1}{q+1}, 1 - \frac{1}{q^2+q}, 1 \}$.
\end{Proposition}
\begin{proof}
  First note that for fixed $q$, only a finite number of
  point weights can occur: Let $P$ be a point. Fix
  an arbitrary $3$-space $\Pi$ through $P$. By Lemma \ref{lem:full_rk},
  the point-line incidence matrix of $\Pi$ has full rank.
  In particular, any of the $2^{q^2+q+1}$
  possible subsets of the $q^2+q+1$ lines of $V$ determines
  the weights on the $q^2+q+1$ points of $V$. In other words,
  at most $\delta := 2^{q^2+q+1} (q^2+q+1)$ weights can occur.

  \medskip

  We will apply Theorem \ref{thm:GLR_proj} and then recursively Theorem \ref{thm:GLR_aff}.
  For this, we need to define several numbers.
  Let $s_1, \ldots, s_{q+1}$ be a sequence of integers which
  satisfies the following:
  \begin{enumerate}[(a)]
   \item We have $s_i \geq 2$ for all $i \in \{ 1, \ldots, q \}$.
   \item Put $s_{i} = \bar{n}_0(s_{i+1}, \delta, q)$
         for $i \in \{ 1, \ldots, q \}$,
         where $\bar{n}_0$ is as in Theorem \ref{thm:GLR_aff}.
   \item Put $s_{q+1} = 2$.
  \end{enumerate}
  Furthermore, put $n_i = s_i + 1$.
  We will show that we can take $n_0(q) = \tilde{n}_0(s_1, \delta, q)$, where
          $\tilde{n}_0$ is as in Theorem \ref{thm:GLR_proj}.
  For instance, $\delta = 896$ for $q=2$, so we find
  $n \geq \tilde{n}_0(\bar{n}_0(\bar{n}_0(2, 896, 2), 896, 2), 896, 2)$.

  \medskip

  Color the points by their weight $w$.
  By Theorem \ref{thm:GLR_proj},
  we find that some $\cP(w)$ contains a subspace $S_1^1$ of dimension $s_1$.
  As $s_1 \geq 2$, $w \in \{ 0, \frac{1}{q+1}\}$.
  As we can replace $Y$
  by its complement, we assume without loss of generality
  that $S_1^1 \subseteq \cP(0)$. Let $P$ be some point not in $S$.
  We want to show that $\wt(P) \in \{ -\frac{q-1}{q}, 0, \frac{1}{q}, 1\}$
  (the other four weights in the assertion occur for
  the complement of $Y$).

  Put $V_1 = \<S_1^1, P\>$,
  and let $T_1$ be some arbitrary $(s_1-1)$-space in $S_1^1$.
  We say that an $n_i$-space $V_i$ has property $C(i)$
  if the following holds:
  \begin{enumerate}[(i)]
   \item $V_i$ contains $P$,
   \item there exist $i$ distinct hyperplanes $S_j^i$ of $V_i$,
   where $j \in \{ 1, \ldots, i \}$,
   with either $S_j^i \setminus S_1^i \subseteq \cP(0)$ or
   $S_j^i \setminus S_1^i \subseteq \cP(\frac1q)$ for any $S_j^i$,
   \item we have $S_1^i \subseteq \cP(0)$,
   \item there exists an $(n_i{-}2)$-space $T_i$ contained in $S_j^i$ for all $j \in \{ 1, \ldots, i \}$,
   \item the point $P$ is not contained in any $S_j^i$ for any $j \in \{ 1, \ldots, i \}$.
  \end{enumerate}
  Note that $V_1$ has property $C(1)$.

  \medskip

  {\bf Claim 1:}
  Let $1 \leq i \leq q-1$.
  If there exists an $n_i$-space $V_i$ with property $C(i)$, then
  there exists an $n_{i+1}$-space $V_{i+1}$ with property $C(i+1)$:

  \medskip

  Figure \ref{fig:rec} illustrates the argument.
  As $\dim(V_i) = n_i$ and $\dim(T_i) = n_i-2$,
  in the quotient space of $T_i$ the hyperplanes in $V_i$ through $T_i$
  correspond to the $q+1$ points on a line.
  That is, there are precisely $q+1$ hyperplanes in $V_i$ through $T_i$.
  Out of these, $i$ corrspond to one of the $S_j^i$ and one
  corresponds to $\< P, T_i \>$. Let $H_i$
  be one of the $q-i$ remaining hyperplanes through $T_i$.
  As $i \leq q-1$, such an $H_i$ exists.
  By applying Theorem \ref{thm:GLR_aff}
  to the $(n_i{-}1)$-space $H_i$
  with $H_i$ in the role of $V$ and $T_i$ in the role of $H$,
  we see that, as $n_i{-}1 = s_i = \bar{n}_0(s_{i+1}, \delta, q)$,
  there exists an $s_{i+1}$-space $S_{i+1}^{i+1}$ in $H_i$
  (but not in $T_i = H_i \cap S_j^i$ for any $j \in \{ 1, \ldots, i \}$)
  with $S_{i+1}^{i+1} \setminus S_1^i \subseteq \cP(w)$ for some $w$.
  By Lemma \ref{lem:aff_weights} with $S_{i+1}^{i+1}$
  for $V$ and $S_{i+1}^{i+1} \cap S_1^{i+1}$ for $H$,
  $w \in \{ 0, \frac{1}{q}\}$.

  Put $V_{i+1} = \< P, S_{i+1}^{i+1}\>$, $T_{i+1} = S_{i+1}^{i+1} \cap T_i$,
  and $S_j^{i+1} = S_j^i \cap V_{i+1}$ for $1 \leq j \leq i$.
  Now $V_{i+1}$ has dimension $n_{i+1} = s_{i+1}+1$,
  and $T_{i+1}$ has dimension $s_{i+1}-1 = n_{i+1}-2$
  and lies in all $S_{i+1}^j$.
  Hence, (i), (iii), and (v) are clear.
  Since $S_j^i$ is a hyperplane of $V_i$,
  $S_j^{i+1}$ is a hyperplane of $V_{i+1}$.
  For $j, j' \in \{ 1, \ldots, i \}$,
  $S_j^{i+1} \cap S_{j'}^{i+1} = S_{i+1}^{i+1} \cap S_j^{i} \cap S_{j'}^{i} = T_{i+1}$.
  Hence, as $\dim(S_j^{i+1}) = n_{i+1}{-}1$ and $\dim(T_{i+1}) = n_{i+1}{-}2$,
  (ii) and (iv) hold for all $j, j' \in \{ 1, \ldots, i \}$.
  By construction, $S_{i+1}^{i+1}$ contains $T_{i+1}$ and is a hyperplane of $V_{i+1}$
  through $T_{i+1}$ distinct from all $S_j^{i+1}$ with $j \in \{ 1, \ldots, i \}$.
  Hence, (ii) and (iv) hold.
  This shows Claim 1.

  \medskip

  Hence, we find an $n_q$-space $V_{q}$ of $V$ which satisfies
  property $C(q)$.

  \medskip

  {\bf Claim 2:} The number of $S_j^{q}$, where $j \in \{ 2, \ldots, q \}$,
  with $S_j^q \setminus S_1^q \subseteq \cP(\frac{1}{q})$
  is either $0$ or $q-1$:

  \medskip

  Let $H_q = \< P, T_q \>$, that is, $H_q$ is the only
  hyperplane of $V_q$ through $T_q$ which is not one of the $S_j^q$.
  Recall that $n_q-1 = s_q \geq \tilde{n}_0(2, \delta, q)$.
  By applying Theorem \ref{thm:GLR_aff} to the $(n_q{-}1)$-space
  $H_q$ with $H_q$ in the role of $V$ and
  $T_q = H_q \cap S_1^q$ in the role of $H$,
  we find a line $S_{q+1}^q$ in $H_q$ (but not in $T_q$)
  with $S_{q+1}^q \setminus S_1^q \subseteq \cP(w)$ for some weight $w$.
  By Lemma \ref{lem:aff_weights} with $S_{q+1}^{q}$
  for $V$ and $S_{i+1}^{q} \cap S_1^{q}$ for $H$,
  $w \in \{ 0, \frac{1}{q}\}$.

  There exists a line $K$ which meets $S_{q+1}^q$
  in a point not in $T_q$ as we can pick a point $R$ of $S_{q+1}^q$
  not in $T_q$ and apply Lemma \ref{lem:pt_and_coline}.
  As $K$ meets $T_q$ trivially,
  the $q+1$ points of $K$ are $S_j^q \cap K$ for $j \in \{ 1, \ldots, q\}$
  and $R = S_{q+1}^q \cap K$. As $\wt(R) \in \{ 0, \frac1q\}$,
  we have $K \subseteq \cP(0) \cup \cP(\frac1q)$.
  Hence, $\wt(K) = \frac{\alpha}{q}$, where $\alpha$ is the number
  of $S_j^q$ with $S_j^q \setminus S_1^q \subseteq \cP(\frac1q)$
  for $j \in \{ 2, \ldots, q+1 \}$.
  As $\wt(K) \in \{ 0, 1 \}$, $\alpha = 0$ and $\alpha=q$ are the only solutions.
  This shows Claim 2.

  \medskip

  By Lemma \ref{lem:pt_and_coline}, there exists a
  line $L$ through $P$ in $V_q$ with no point in $T_q$.
  The $q+1$ points of $L$ are $S_j^q \cap K$
  for $j \in \{ 1, \ldots, q\}$ and $P$.
  Let $\beta$ denote the number of  $j \in \{ 1, \ldots q \}$ with
  $S_j^q \setminus S_1^q \subseteq \cP(\frac1q)$.
  Then $\wt(L) = \frac{\beta}{q} + \wt(P)$.
  By Claim 2, $\beta = 0$ or $\beta=q-1$.
  If $\beta = 0$ and $\wt(L) = 0$, then $\wt(P) = 0$.
  If $\beta = 0$ and $\wt(L) = 1$, then $\wt(P) = 1$.
  If $\beta = q-1$ and $\wt(L) = 0$, then $\wt(P) = -\frac{q-1}{q}$.
  If $\beta = q-1$ and $\wt(L) = 1$, then $\wt(P) = \frac{1}{q}$.
  Thus, $P$ has one of the allowed weights.
\end{proof}

\begin{proof}[Proof of Theorem \ref{thm:main}]
  As mentioned earlier, we only need to show the theorem for $k=2$.
  By Proposition \ref{prop:rec_ram}, if $n \geq n_0(q)$, then
  all point weights are in $\{ -\frac{q}{q+1}, -\frac{q-1}{q}, -\frac{1}{q^2+q}, 0, \frac{1}{q}, \frac{1}{q+1}, 1 - \frac{1}{q^2+q}, 1 \}$. Hence, we can apply Lemma \ref{lem:classical_weights_are_good}
  which shows the assertion.
\end{proof}

\section{Two-Intersection Sets}

In the 1970s Tallini Scafati investigated sets of type
$(\alpha, \beta)_{k-1}$, see \cite{TS1972,TS1976}.
A set of type $(\alpha, \beta)_{k-1}$, $\alpha < \beta$,
in the projective space $\PG(n-1, q)$ is a
set of points $\cP$ such that each $k$-space
intersects $\cP$ in either $\alpha$ or $\beta$ points.
In particular, unless $\cP$ is a point,
a hyperplane, or the complement of either of these,
$q$ has to be an odd square for $n \geq 4$ and $k=2$, see XIII) in \cite{TS1972}.
In this case there are essentially two choices for
$(\alpha, \beta)$. The existence of such sets
is unknown. Here we settle existence asymptotically for fixed $k$.

\begin{Theorem}\label{thm:two_int}
  Let $q$ be a prime power and let $k$ be an
  integer greater or equal to $2$.
  There exists a constant $n_0(k, q)$ such
  that any set of type $(\alpha, \beta)_{k-1}$
  in $\PG(n-1, q)$
  has either $\alpha = 0$ or $\beta = q+1$
  and is either a point, a hyperplane,
  the complement of a point, or
  the complement of a hyperplane.
\end{Theorem}
\begin{proof}
  Put weights on $\cP$ and its complement
  such that $\wt(S) = 0$ for $k$-spaces
  $S$ which intersect $\cP$ in $\alpha$
  points and $\wt(S) = 1$ for $k$-spaces
  which intersect $\cP$ in $\beta$ points.
  Then the assertion follows from Theorem \ref{thm:main}.
\end{proof}

Alternatively, one can use Theorem \ref{thm:GLR_proj} directly.
Note that a similar result holds true for affine
two-intersection sets using Theorem \ref{thm:GLR_aff}.

\section{Small Dimension}

Blokhuis, De Boeck, and D'haeleseer showed that
a non-trivial Boolean degree $1$ function $Y$ of size $x \gauss{n-1}{k-1}$
has $x > C q^{n/2 - k + 1/2}$,
see \cite{BDBD2019} for their precise statement.
Later De Beule, Mannaert, and Storme proved
$x > C q^{n-3k}$ in \cite{DBMS2022}.
Here we improve this to $x > C q^{n-3k+2}$
using Proposition \ref{prop:wt_bnds_gen}.%
\footnote{After the submission of this document, De Beule, Mannaert, and Storme announced
a further improvement of $x > C q^{n - 2.5k + 0.5}$ for $n \geq 3k+2$ in \cite{DBMS2024}.}

\begin{Theorem}\label{thm:bnd_x}
  Let $Y$ be a Boolean degree $1$
  function on $J_q(n, k)$ of size $x \gauss{n-1}{k-1}$.
  If $Y$ is not Example (I) or Example (II), then
  \[
   x \geq \frac{[n-1]}{q[k-1]^2[k]} > \frac{1}{8} q^{n - 3k + 2}.
  \]
\end{Theorem}
\begin{proof}
  Without loss of generality assume that $Y$ does not contain a point-pencil
  (otherwise, remove the point-pencil from $Y$).
  A $k$-space $S$ in $Y$ has weight $1$, so we find one point $P$ in $S$
  with $\wt(P) \geq \frac{1}{[k]}$. As $Y$ does not contain point-pencils,
  the point $P$ must lie on another $k$-space with weight $0$. Hence,
  the average weight of the points of $S$ minus $P$ is $\frac{-1}{q[k-1][k]}$.
  Let $Q$ be a point with at most this weight.
  Then the first inequality in Proposition \ref{prop:wt_bnds_gen}
  shows the first inequality in the assertion. The second inequality
  follows from $q^{a-1} < [a] < \frac{q}{q-1} \cdot q^{a-1} \leq 2q^{a-1}$.
\end{proof}

Example (III) has $x \sim C q^{n-2k}$.
Hence, the argument is tight for $k=2$, $n$ fixed and $k \rightarrow \infty$.

\section{On Mazzocca and Tallini} \label{sec:mt}

Let $L \in R\cL(q)$.
Let $\Pi = \{ \Pi_1, \ldots, \Pi_\ell \}$ a partition
of the $d$-spaces (or rank $d$ elements) of $L$.
We say that $\Pi$ is a {\it $(k, d, \ell; q)$-blocking set partition}
if any $k$-space in $L$ contains at least one element
of each $\Pi_j$. Mazzocca and Tallini proved in \cite{MT1985}
that $(k, 1, 2; q)$-blocking set partitions do not exist
for $n \geq \tilde{n}_0(k, q)$ for some $\tilde{n}_0(n, q)$.
Following the same argument
as in \cite{MT1985} using Theorem \ref{thm:voigt_ramsey},
we find the following result.

\begin{Theorem}
  Let $q$ be a prime power. For every pair of positive integers $k$ and $\ell$,
  there exists an $n_0(k, \ell, q)$ such that
  no lattice $L \in R\cL(q)$ with $\rank(L) \geq n_0(k, \ell, q)$
  possesses a $(k, d, \ell; q)$-blocking set partition.
\end{Theorem}

\section{Future Work}

Let us conclude with some suggestions for future work.
\begin{itemize}
 \item Determine the smallest value for $n_0(k, q)$ for which $B(n_0(k, q), k, q)$ holds.
      The implicit bound in our proof presented here on $n_0(k, q)$ is very large,
      while it is known that $n_0(2, 2) = 4$ and $n_0(2, q) = 5$
      for $q \in \{ 3, 4, 5 \}$. There has been some recent
      interest in giving good asymptotic bounds on $q$-analogs
      of Ramsey numbers, see \cite{FY2023}. Maybe these efforts
      can lead to better estimates for $n_0(k, q)$.
      Disproving $B(5, 2, q)$ for any $q$
      would also be very interesting.
 \item Show a Friedgut-Kalai-Naor (FKN) theorem, see \cite{FKN2002}, for vector spaces.
 That is, show that an almost Boolean degree $1$ function is
 close to a union of Boolean degree $1$ functions, cf. \cite{Filmus2016}. The problem
 for our technique used here is that the number of possible weights is no longer finite.
 \item Extend the argument to Boolean degree $d$ functions on $J_q(n, k)$.
 Here one puts weight on $d$-spaces instead of $1$-spaces. Note that
 such a characterization of low degree Boolean functions is known
 for the Johnson scheme $J(n, k)$, see \cite{Filmus2023,FI2019a}.
 Using Theorem \ref{thm:voigt_ramsey}, a variant of Proposition \ref{prop:rec_ram}
 might be still feasible, but there are no geometrical results
 comparable to Proposition \ref{prop:drudge_easy_ind}
 and Proposition \ref{prop:metsch_bnds}.
 \item Extend the argument to closely related structures,
 for instance bilinear forms or the class $R\cL(q)$.
\end{itemize}

\paragraph*{Acknowledgments}
The author thanks two referees, Anurag Bishnoi, Jozefien D'haeseleer,
Sasha Gavrilyuk, Jonathan Mannaert, Morgan Rodgers, and, in particular, Yuval Filmus for their comments
related to this document.
We thank Lukas Kühne and Hans Jürgen Prömel for helping with
access to \cite{Voigt1978}.


\end{document}